\documentclass {amsart}
\usepackage{amssymb,amsthm,amsmath}
\usepackage{amsmath,amssymb,latexsym}
\usepackage{color}
\usepackage{epsfig}

\newcommand{\glp}{{\ensuremath{\textup{\textbf{GLP}}}}\xspace}
\newcommand{\pair}[1]{{\ensuremath{\langle #1 \rangle}}\xspace}

\theoremstyle{plain}
\newtheorem{theorem}{Theorem}[section]
 \newtheorem{lemma}[theorem]{Lemma}
 
 \newtheorem{conjecture}[theorem]{Conjecture}
 \newtheorem{corollary}[theorem]{Corollary}
 \newtheorem{corol}[theorem]{Corollary}
 \newtheorem{fact}[theorem]{Fact}
 
 \newtheorem{remark}[theorem]{Remark}
  \newtheorem{pro}[theorem]{Proposition}

 \newtheorem{definition}[theorem]{Definition}

\newcommand{\s}{{\ensuremath {\tt S}}\xspace}
\usepackage{xspace}
\newcommand{\extil}[1]{\ensuremath{\textup{\textbf{IL}}{\sf\ensuremath{#1}}}\xspace}
\newcommand{\ilw}{\extil{W}}
\newcommand{\restintl}[2]{{\ensuremath {\textup{\textbf{IL}}}_{#2}({\rm #1})}}
\newcommand{\pra}{\ensuremath{{\mathrm{PRA}}}\xspace}
\newcommand{\il}{{\ensuremath{\textup{\textbf{IL}}}}\xspace}
\newcommand{\formil}{\ensuremath{{\sf Form}_{\il}} \xspace}
\newcommand{\gl}{{\ensuremath{\textup{\textbf{GL}}}}\xspace}
\newcommand{\formgl}{\ensuremath{{\sf Form}_{\gl}} \xspace}
\newcommand{\intl}[1]{{\ensuremath {\textup{\textbf{IL}}}({\rm #1})}}
\newcommand{\principle}[1]{\formal{#1}}
\newcommand{\formal}[1]{\ensuremath{{\sf {#1}}\xspace}}
\newcommand{\ir}[1]{{\ensuremath{\mathrm{I}\Sigma^R_{#1}}}\xspace}
\newcommand{\isig}[1]{{\ensuremath {\mathrm{I}\Sigma_{#1}}}\xspace}
\newcommand{\ilp}{\extil{P}}
\newcommand{\ea}{\ensuremath{{\rm{EA}}}\xspace}
\newcommand{\ilm}{\extil{M}}

\begin{document}

\author{Thomas F. Icard}
\address{Stanford University}
\email{icard@stanford.edu}

\author{Joost J. Joosten}
\address{Universidad de Sevilla}
\email{jjoosten@us.es}

\title{Provability and Interpretability Logics with Restricted Realizations}

\maketitle 

\begin{abstract}
The provability logic of a theory $T$ is the set of modal formulas, which under any arithmetical realization are provable in $T$. We slightly modify this notion by requiring the arithmetical realizations to come from a specified set $\Gamma$. We make an analogous modification for interpretability logics. This is a paper from 2012. 

We first studied provability logics with restricted realizations, and show that for various natural candidates of theory $T$ and restriction set $\Gamma$, where each sentence in $\Gamma$ has a well understood (meta)-mathematical content in $T$, the result is the logic of linear frames. However, for the theory Primitive Recursive Arithmetic (\pra), we define a fragment that gives rise to a more interesting provability logic, by capitalizing on the well-studied relationship between \pra and \textsf{I}$\Sigma_1$.

We then study interpretability logics, obtaining some upper bounds for \intl{\pra}, whose characterization remains a major open question in interpretability logic. Again this upper bound is closely relatively to linear frames. The technique is also applied to yield the non-trivial result that $\intl{\pra}\subset \ilm$.

\end{abstract}



\section{Introduction} 
In a recent discussion on a mailing list on the foundations of mathematics\footnote{FOM mailing list, 8/19/2009, 
{\tt http://cs.nyu.edu/pipermail/fom/2009-August/013994.html }
} Joe Shipman asked for important theorems that have essentially only one proof. In reply, Giovanni Sambin provided the example of Solovay's arithmetical completeness theorem  of the provability logic \textbf{GL} (\cite{Sol76}). 

This paper deals with restricted cases of Solovay's theorem where alternative proof-methods are available. One of the broad motivations for this paper is the hope of obtaining an alternative proof of Solovay's Theorem (see Section \ref{section:Solovay}). However, the method of provability logics with restricted realizations, we feel, merits interest in its own right, as we shall explain shortly. Let us first briefly restate Solovay's Completeness Theorem, which is the cornerstone result in the study of provability logics.

\subsection{Provability Logics}

The propositional modal logic \textbf{GL}, known as \textit{G\"{o}del-L\"{o}b Logic},  captures exactly the behavior of the standard provability predicate in arithmetic. For a given theory $T$ (e.g. Peano Arithmetic), formulas $\Box A$ are interpreted as, ``$A$ is provable in $T$". It is defined by extending the basic modal logic \textbf{K} with a schematic formalization of L\"{o}b's Theorem (\textbf{L} in the following definition).
\begin{definition} \textnormal{\textbf{GL}} is given by all boolean tautologies, in addition to all instances of the following schemata
\[
\begin{array}{lll}
\mbox{\textnormal{\textbf{K}}} &:& \Box (A\to B) \to (\Box A \to \Box B);\\
\mbox{\textnormal{\textbf{L}}} &:& \Box (\Box A\to A) \to \Box A.\\
\end{array}
 \]  The logic is closed under \textit{modus ponens} and necessitation. \end{definition}
\textbf{GL} enjoys modal completeness with respect to a simple class of frames, in particular the class of finite, irreflexive, and transitive frames, which we henceforth refer to as \textbf{GL}-frames. The logic is linked to formalized provability via \emph{arithmetical realizations}. An arithmetical realization is a function $*$ that maps propositional variables to sentences in the language of (a given) arithmetic, sending $\bot$ to $0=1$. A realization $*$ can be extended uniformly so that we can interpret an arbitrary modal formula as an arithmetical formula by stipulating,
\[
\begin{array}{l}
(A\to B)^* = A^* \to B^*\\
(\Box A)^* = {\sf Bew}_T(\ulcorner A^* \urcorner).
\end{array}
\]
Here $\ulcorner \cdot \urcorner$ is a function that maps a formula $\varphi$ to its code $\ulcorner \varphi \urcorner$ and ${\sf Bew}_T(\cdot)$ is a predicate in the language of $T$ formalizing provability in $T$, so that $T\vdash \varphi$ just in case $\mathbb{N} \vDash  {\sf Bew}_T(\ulcorner \varphi \urcorner)$.

We define \textbf{PL}($T$), the \emph{provability logic} of a theory $T$, as follows
\[
\mbox{\textbf{PL}($T$)} := \{ A \mid \forall * T \vdash A^* \}.
\] 
Since L\"ob (\cite{Loeb55}) it is known that \textbf{GL} is sound for a large class of theories $T$, that is,  \textbf{GL} $\subseteq$ \textbf{PL}(\textsf{T}). The reverse inclusion is Solovay's completeness result.

\begin{theorem}[Solovay's Theorem]\label{theorem:solovay}
\textnormal{\textbf{PL}(}$T$\textnormal{)} = \textnormal{\textbf{GL}} for a wide range of theories $T$.
\end{theorem}

For soundness, i.e.\ \textbf{GL}$\subseteq$\textbf{PL}$(T)$, the theory can be as weak as $\textsf{I}\Delta_0 + \Omega_1$ or equivalently Buss's ${\sf S^1_2}$ (see \cite{BerarducciVerbrugge} and \cite{Buss86}).
Arithmetical completeness, i.e. that \textbf{PL}${(T)\subseteq}$\textbf{GL}, is known to hold for any sound\footnote{$\Sigma_1$-sound is sufficient here.\label{footnote:ClassOfTheories}} theory extending $\textsf{I}\Delta_0 + \exp$ (see \cite{JongJumeletMontagna}).

Solovay proved that whenever \textbf{GL} $\nvdash A$, there is a realization $*$ so that \textsf{PA} $\nvdash A^*$. An outline of the proof runs as follows. First, a modal countermodel $\mathcal{M}$ in the form of a rooted tree is taken that witnesses \textbf{GL} $\nvdash A$. Next, a new root is added to this model. A primitive recursive function $f$ on this model is defined in terms of its own provable limit behavior. This definition is made using an arithmetical  fixed point. The function $f$ starts in the newly added root and $f(x)$ remains where it is unless $x$ is a proof that the function does not have the node $y$, which is accessible from $x$, as a limit, in which case the function jumps to $y$. If $T$ is a sound theory, the function must stay where it started, in the newly added root. The realization $*$ is defined as a disjunction of the limit-statements $\lambda_y$ of the function $f$, where $\lambda_y$ says ``$y$ is the limit of $f$". More specifically  $p^*:= \bigvee_{\mathcal{M}, y \Vdash p}\lambda_y$.

\subsection{Restricted Realizations}
This ingenious proof thus gives us the concrete realization $*$. However, the arithmetical content of this realization $*$ is not exactly transparent.\footnote{There is a paper by de Jongh, Jumelet and Montagna \cite{JongJumeletMontagna} where an alternative proof of Solovay's theorem is given. In that proof, using the diagonal lemma, one finds some sentences with the required properties rather than defining the sentences and then proving the necessary properties. However, the obtained sentences are essentially the same as the ones defined in Solovay's original proof. } A natural question to ask is whether we can find translations with more clear arithmetical and proof theoretic content. And conversely, given a set of arithmetical sentences with a clear arithmetical content, what modal logics results from restricting realizations to this particular set? These questions motivate the following definition. We shall write, \emph{par abus de langage}, $*\in \Gamma$ to mean that the realization $*$ takes on all its values within the set of sentences $\Gamma$.

\begin{definition}
\textnormal{\textbf{PL}}$_{\Gamma}(T) := \{ A : \forall *\in \Gamma , \ T\vdash A^* \}$
\end{definition}
Notice that in a strict sense, \textnormal{\textbf{PL}}$_{\Gamma}(T)$ need not even be a logic as in general it is not closed under substitution. From the definition the following lemma is evident.
\begin{lemma}\label{lemma:upperboundmethod}
If $\Gamma \subseteq \Delta$, then \textnormal{\textbf{PL}}$_{\Delta}(T) \ \subseteq$ \textnormal{\textbf{PL}}$_{\Gamma}(T)$.
\end{lemma}
Clearly, by taking $\Gamma$ to be the set of \emph{all} arithmetical sentences we get  \textbf{PL}$_{\Delta}(T)=$ \textbf{PL}$(T)$. For a large class of theories however, we can improve this to the following theorem.

\begin{theorem}\label{theorem:BoolSigmaPL}
For all those theories $T$ for which Solovay's Theorem \ref{theorem:solovay} can be proved using the original proof we have that
\[
\textnormal{\textbf{PL}}_{\mathcal{B}(\Sigma_1)}(T) \ \subseteq \textnormal{\textbf{PL}}(T).
\]
Here, $\mathcal{B}(\Sigma_1)$ denotes the class of Boolean combinations of $\Sigma_1$ sentences.
\end{theorem}

\begin{proof}
By close inspection of the proof of Solovay's theorem, we see that all substitutions are disjunctions of limit statements.  It is clear that for elementary functions $h$, the statement ``$h$ has a limit" can be expressed in a $\Sigma_2$ fashion. However, as Albert Visser pointed out to us, the statement ``$h$ has limit $i$" which is only actually needed in Solovay's proof, can be expressed as $\mathcal{B}(\Sigma_1)$:
\[
\left[ \exists x\; h(x)=i\right] \ \ \wedge \ \ \left[ \forall y,z \; ((y\leq z \wedge hy=i) \to h(z) = i)\right].
\]
Taking disjunctions of those sentences will of course not get us out of the class $\mathcal{B}(\Sigma_1)$.
\end{proof}

The simple Lemma \ref{lemma:upperboundmethod} can be used to establish upperbounds for a provability logic if one is unable to find the full provability logic. For example, it is a long standing open question what the provability logic is of bounded arithmetics such as $\sf S^1_2$.\footnote{This question has been studied in depth in \cite{BerarducciVerbrugge}. The question also has important connections to matters in computational complexity. For example, it is shown in \cite{Buss86} that if $\sf S^1_2$ proves $\Pi_1^b$-completeness with parameters ($\Pi_1^b$ is the set of formulas $(\forall x \leq t)\  \theta$ with $\theta$ sharply bounded), then {\sf NP = coNP}.} 

\subsection{Applications and plan of the paper}\label{subsection:ApplicationsAndPlan}

One can thus use $\Gamma$ to study the provability logic of $T$. On the other hand, we shall see that \textbf{PL}$_{\Gamma}(T)$ can also be used to \emph{characterize}  the fragment $\Gamma$. For example, in Theorem \ref{theorem:mainTheorem}  below we consider the closed fragment $\mathcal{B}$ of provability logic, which consists of boolean combinations of iterated (in)consistency statements. This fragment is given by the following grammar.
\[
\mathcal{B}\ \ \ := \ \ \ \bot \ \mid \ \mathcal{B}\to \mathcal{B}   \ \mid \ \Box \mathcal{B}.
\]
We shall see that the modal formulas valid under all realizations from this fragment are exactly the formulas valid on all finite strict linear orders. This can be said to provide yet further evidence that reflection principles and iterated consistency statements are inherently linearly ordered. 

Moreover, this fact also gives us information on what kind of arithmetical fixed point constructions are needed in the proof of Theorem \ref{theorem:solovay}. By the modal Fixed Point Theorem, independently due to de Jongh and Sambin (see \cite{Sambin} (de Jongh actually never published his proof)), we know that certain applications of the arithmetical fixed point theorem can be dispensed with. More precisely, if we have a  formula $A(x)$ where the $x$ only occurs directly under the scope of a ${\sf Bew}_T$ predicate then applying the fixed point to this formula does not give us new expressive power. That is, if we can prove $B \leftrightarrow A(\ulcorner B \urcorner)$ then $B$ is actually provably equivalent to a formula in the language of provability logic.
Thus, these sort of applications of the arithmetical fixed point theorem only yield formulas in $\mathcal{B}$, whence, \textit{pace} Theorem \ref{theorem:closedFragment}, cannot suffice for a proof of Solovay's completeness result, Theorem \ref{theorem:solovay}.

Another example of restricting the substitutions is known in the literature. In \cite{viss:81} Visser studied the provability logic that arises when restricting substitutions to $\Sigma_1$ sentences. 

\medskip

In Section \ref{section:glp} we shall consider a fragment $\mathcal{D}$ which contains infinitely many copies of $\mathcal{B}$ for increasingly strong provability predicates. It turns out that even for this richer fragment we do not move beyond linear frames (cf. Theorem \ref{theorem:ClosedFragmentGLP}). However, in Section \ref{section:NonLinear} we shall see that there is a natural fragment for \textsf{PRA} whose associated provability logic lies strictly in between the logic of linear frames and \textbf{GL}.

In Section \ref{section:InterpretLogics} we shall see how restricted realizations can also be applied to interpretability logics.

\section{Fragments and Logics} \label{fragmentslogics}


In this section we show that certain conditions on a given fragment translate to a semantic characterization of the corresponding restricted provability logics. First, some preliminaries on basic relational semantics for provability logics.

Recall that a \textit{frame} $\mathbb{F}$ for \textbf{GL} is an ordered pair $\langle W, R \rangle$, where $W$ is a set of points and $R \subseteq W \times W$ is a finite, irreflexive, and transitive relation. Given a set \textsf{Prop} of propositional variables, a \textit{model} $\mathcal{M}$ \textit{based on $\mathbb{F}$} is a triple $\langle W,R,V \rangle$, where $V: \textsf{Prop} \rightarrow \wp(W)$ is a \textit{valuation function} assigning to each variable the set of the points where it is true. We shall also write $V$ for the straightforward extension of $V$ to arbitrary modal formulas. We then write $\langle W,R,V \rangle , w \vDash A$, just in case $w \in V(A)$. We write $\langle W,R,V\rangle \vDash A$ if $A\in V(w)$ for all $w\in W$. Overloading notation, we also write $\langle W,R \rangle \vDash A$, if $\langle W,R,V \rangle \vDash A$ for all $V$. We say $A$ is \textit{valid} in the model and in the frame, respectively. 



When dealing with fragments, however, arbitrary variables will not be present. All of the fragments we shall consider in this paper will extend the fragment $\mathcal{B}$ defined above, by adding constants $\sigma _1, \sigma _2, \sigma _3,...$, with some clear arithmetical content. As these constants will be fixed, and as we would like to characterize the sentences in this fragment modally, we shall add constants $s_1, s_2, s_3,...$, to our modal language, and correspondingly extend the definition of a realization to ensure that $(s_i)^* = \sigma _i$. In fact, given this convention, we will be able to define our fragments in a single language and throughout treat each constant simultaneously as a constant in the modal language and as a specified arithmetical formula, disambiguating whenever the distinction is not clear from context. In other words, we will usually not distinguish between $A$ and $A^*$.

On the other hand, as far as the relational semantics is concerned, the constants $s_1, s_2, s_3,...$ are simply treated as variables. Therefore the above notation is extended in the obvious way to this setting.

Suppose we would like to obtain a modal characterization of \textbf{PL}$_{\mathcal{F}}(T)$. Under certain circumstances, it suffices to know how $\mathcal{F}$ is characterized according to $T$. To be precise, if we have a model $\mathcal{M}$ based on a frame $\mathbb{F}$, such that for each $A \in \mathcal{F}$, the following condition holds:
\begin{equation}\label{equation:FroedelCondition}
T \vdash A \ \Leftrightarrow \ \mathcal{M} \vDash A,
\end{equation}
then, as is shown in Theorem \ref{theorem:mainTheorem} below, \textbf{PL}$_{\mathcal{F}}(T) = \mathcal{L}(\mathbb{F})$. Here, $\mathcal{L}(\mathbb{F})$ is the set of formulas in the basic modal language (with propositional variables) valid on the frame $\mathbb{F}$.

There are two side conditions to our theorem. One of them involves image-finiteness. We call a model \emph{image-finite} if $\{y : xRy\}$ is finite for each $x$. We shall denote the set $\{y : xRy\} \cup \{x\}$ by $x\uparrow$.
Our theorem thus reads as follows:

\begin{theorem}\label{theorem:mainTheorem}
Suppose that \eqref{equation:FroedelCondition} holds for a model $\mathcal{M}$ based on frame $\mathbb{F}$. Suppose moreover that $\mathcal{M}$ is image-finite and that each point $x\in \mathcal{M}$ is uniquely definable by a formula $D_x \in \mathcal{F}$. Then, we have that \textnormal{\textbf{PL}}$_{\mathcal{F}}(T) = \mathcal{L}(\mathbb{F})$.
\end{theorem}

\begin{proof}
In the light of \eqref{equation:FroedelCondition} it suffices to prove that 
\[
\forall \, *\, \in \mathcal{F},\  \mathcal{M}\vDash B^*\ \Leftrightarrow \  \mathbb{F} \vDash B.
\]
\begin{itemize}
\item[$\Leftarrow$]
Consider some arbitrary $* \in \mathcal{F}$ and define $V_*(p) := \{ i : \mathcal{M},i \vDash p^*\}$. By induction on $A$ we see that for each $i \in \mathbb{F}$ 
\[
\langle \mathbb{F}, V_* \rangle, i \Vdash A \ \Leftrightarrow \ \mathcal{M}, i \Vdash A [p/p^*]
\]
and we are done.

\item[$\Rightarrow$]
Given some $i \in \mathbb{F}$ and some arbitrary valuation $V$ we define $*$ by
\[
p^* := \bigvee_{x\in V(p) \cap i\uparrow} D_x.
\]
As the frame is image-finite, the disjunction is finite. By an induction\footnote{In order to get the inductive step for the $\Box$ operator going we should prove the slightly stronger statement that for all $j\in i\uparrow$ we have $\langle \mathbb{F}, V \rangle, j \vDash C \ \Leftrightarrow \ \mathcal{M}, j \vDash C^*$.} on $C$ we see again that
\[
\langle \mathbb{F}, V \rangle, i \vDash C \ \Leftrightarrow \ \mathcal{M}, i \vDash C^*.
\]
As $i$ was arbitrary, we see that $\mathbb{F} \vDash C$.
\end{itemize}

\end{proof}

As we shall see below, in many occasions we actually will have something stronger than \eqref{equation:FroedelCondition}. In particular we shall often find ourselves in a situation where we have, apart from the frame, also 
 a modal logic \textbf{L} for which we have
\[
T \vdash A \ \Leftrightarrow \ {\bf L} \vdash A \ \Leftrightarrow \ \mathcal{M} \vDash A.
\]
This logic \textbf{L} will facilitate our calculations considerably.

\section{The Closed Fragment} \label{section:closedfragment}
With Theorem \ref{theorem:mainTheorem} we can calculate our first provability logic with restricted substitutions. Recall the definition of the closed fragment $\mathcal{B}$ in Subsection \ref{subsection:ApplicationsAndPlan}. 
\begin{definition}
\textnormal{\textbf{GL.3}} is the logic \textnormal{\textbf{GL}} together with the linearity axiom:
\[
\Box(\Box A \to B) \vee \Box (\Box^+ B \to A).
\]
Here and below, $\Box^+ A$ is short for $A \wedge \Box A$.
\end{definition}

\begin{theorem}\label{theorem:closedFragment}
\textnormal{\textbf{PL}$_{\mathcal{B}}$($T$)}= \textnormal{\textbf{GL.3}} for a large class\footnote{See Footnote \ref{footnote:ClassOfTheories} on conditions on theories. The current proof of this theorem invokes Solovay's completeness result, Theorem \ref{theorem:solovay}, in full. However, in \cite{Joo03} it is shown how we can substitute the use of Solovay's completeness result by the proof of Theorem \ref{theorem:mainTheorem}. Thus, Theorem \ref{theorem:closedFragment} actually holds for a larger class of theories including $\textsf{I}\Delta_0 + \Omega_1$.} of theories $T$.
\end{theorem}

\begin{proof}
It is well known that the truth of a closed formula at a particular point in a model depends solely on the rank of that point. Here, the rank of a point $x$ is defined as the supremum of lengths of paths leading from $x$ to a leaf. See for example Chapter 7 from \cite{Bool93}.

Thus, the linear frame $\langle \omega, > \rangle$ is universal for $\mathcal{B}$ in the sense that if a formula $\varphi \in \mathcal{B}$ is false at some point in some frame, then it is actually false at some point in $\langle \omega, > \rangle$. Thus, by Theorem \ref{theorem:solovay}, we have 
$T \vdash A \ \Leftrightarrow \ \langle\omega,>\rangle \vDash A$.

Furthermore, it is known that the logic of the frame $\langle \omega, > \rangle$ is axiomatized by  \textbf{GL.3}. (See, for example, Chapter 13 of \cite{Bool93}.) Thus, $\langle\omega,>\rangle \vDash A\ \Leftrightarrow \ \mbox{\textbf{GL.3}}\vdash A$ and Condition \ref{equation:FroedelCondition} is satisfied for any model based on $\langle\omega,>\rangle$.

Note that $\langle\omega,>\rangle$ is image-finite and that the point $n$ is defined by $\Diamond^n \top \wedge \Box^{n+1}\bot$. Thus, by Theorem \ref{theorem:mainTheorem} we have our result.
\end{proof}


\section{Substitutions from the Closed Fragment of \textbf{GLP}}\label{section:glp}
Japaridze's Logic \textbf{GLP} (\cite{Jap86}) describes all of the universally valid schemata for reflection principles of restricted logical complexity in arithmetic. It is formulated in a language with infinitely many modalities, where $[n]A$ is read arithmetically as, \begin{center} \emph{$A$ is provable from $T$ along with all true $\Pi_n$ sentences.} \end{center} Arithmetical completeness with respect to this interpretation was proven in \cite{Ign93}, for sound theories containing only a modest amount of arithmetic.
\begin{definition} \textnormal{\textbf{GLP}} is given by the following axiom schemata, \begin{enumerate}
\item[(i)] All boolean tautologies;
\item[(ii)] $[n]([n]A \rightarrow A) \rightarrow [n]A$, for all $n$;
\item[(iii)] $[m]A \rightarrow [n]A$, for $m \leq n$;
\item[(iii)] $\langle m \rangle A \rightarrow [n]\langle m \rangle A$, for $m<n$;
\end{enumerate}
in addition to the rules of \textit{modus ponens} and necessitation for each $[n]$.
\end{definition}
While \textbf{GLP} does not admit of any frame semantics, various other models have been given (see, e.g. \cite{Bek09} and \cite{BBI09}). In particular, Ignatiev \cite{Ign93} has defined a \textit{universal frame} for the closed fragment of \textbf{GLP}, denoted \textbf{GLP}$_0$, which will be of use.\footnote{This frame is studied in detail in \cite{BJV} and \cite{Ica09}.}

Define $\mathcal{D}$ to be the fragment given by the following infinite grammar:
\[\mathcal{D} \; := \; \bot \; \vert \; \mathcal{D} \rightarrow \mathcal{D} \; \vert \; [0]\mathcal{D} \; \vert \; [1]\mathcal{D} \; \vert \; [2]\mathcal{D} \; \vert \; ...\] That is, \textbf{GLP}$_0$ is simply \textbf{GLP} restricted to the fragment $\mathcal{D}$, with no variables.

We can describe Ignatiev's universal frame for \textbf{GLP}$_0$ as follows. Let $\Omega$ consist of the set of $\omega$-sequences of ordinals $(\alpha _0,\alpha _1, \alpha _2,...)$, where each $\alpha _i < \epsilon _0$. Recall $\epsilon _0$ is the least fixed point of the equation $\omega ^{\alpha} = \alpha$. If the Cantor Normal Form of $\alpha$ is $\omega ^{\lambda _n} +...+ \omega ^{\lambda _1}$, then let $e(\alpha) := \lambda _1$ and set $e(0) = 0$.
\begin{definition} Ignatiev's universal frame is defined as $\mathcal{U} := \langle U, \{R_n\}_{n<\omega}\rangle$, with, \[U \; := \; \{\vec{\alpha} \in \Omega: \forall i < \omega, \alpha _{i+1} \leq e(\alpha _i)\};\]
\[\vec{\alpha}R_n\vec{\beta} \; :\Leftrightarrow \; (\forall m < n, \alpha _m = \beta _m \; \& \; \alpha _n > \beta _n).\]
\end{definition}
Notice that each point in $U$ can be seen as a finite, strictly decreasing sequence of ordinals less than $\epsilon _0$, as each sequence ends in an infinite tail of zeros. For a visualization of the frame, see Figure \ref{fig:ignatiev}.

\begin{figure}
\begin{center}
\input{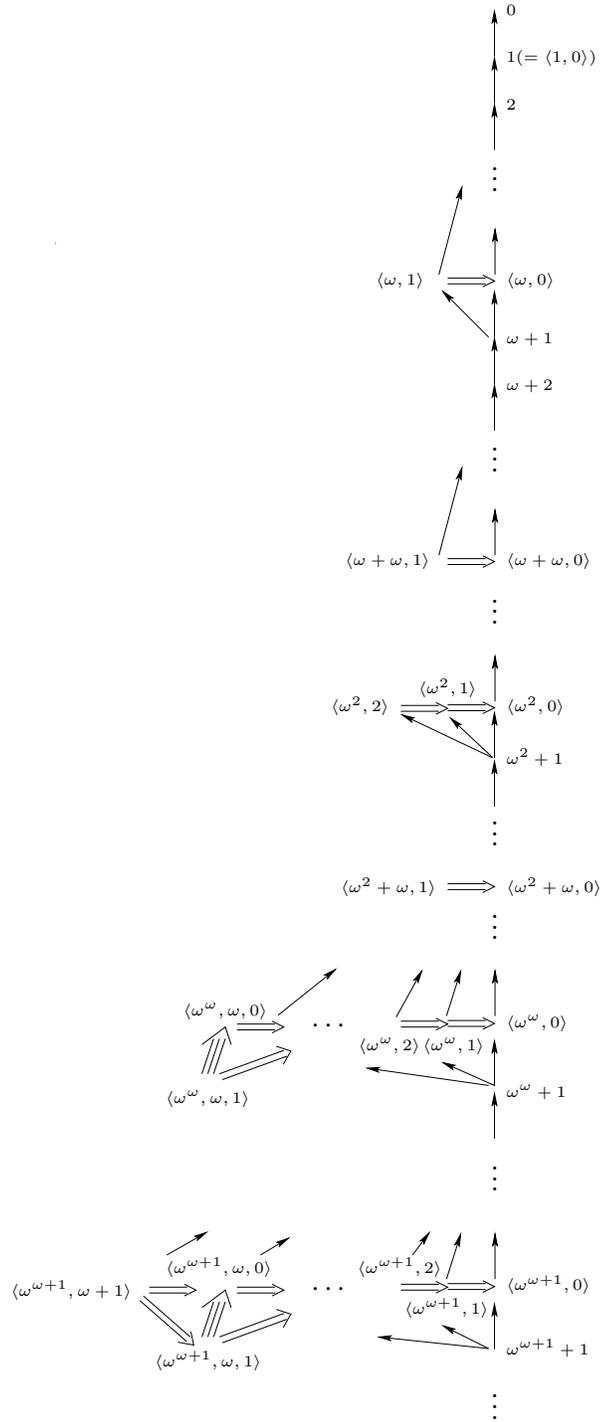}
\caption{The universal model for $\glp_0$} \label{fig:ignatiev}
\end{center}
\end{figure}

A point of the form $(\alpha,e(\alpha),e(e(\alpha)),...)$, where $\alpha _{i+1} = e(\alpha _i)$ for all $i$, is called a \textit{root point}, and is denoted by $\widehat{\alpha}$ when $\alpha$ is the first coordinate. Thus every coordinate of $\widehat{\alpha}$ is uniquely determined by $\alpha$. The following lemma is then obvious, given the definition of $\mathcal{U}$.
\begin{lemma} \label{compar} If $\widehat{\alpha}$ and $\widehat{\beta}$ are root points, then either $\widehat{\alpha} R_0 \widehat{\beta}$, $\widehat{\beta}R_0\widehat{\alpha}$, or $\widehat{\alpha} = \widehat{\beta}$.
\end{lemma}
In addition to the more routine soundness, the following strong completeness theorem has also been proven using several different methods in the works cited above.
\begin{theorem} \label{arithcom} If \textnormal{\textbf{GLP}}$_0 \nvdash A$, then there is a root point $\widehat{\alpha} \in U$, such that $\mathcal{U},\widehat{\alpha} \nvDash A$.
\end{theorem}
With these results we can now show that even with this much richer fragment the resulting provability logic is exactly the same as for the fragment with only the single $\Box$-operator (c.f. Theorem \ref{theorem:closedFragment}).
\begin{theorem} \label{theorem:ClosedFragmentGLP}
\textnormal{\textbf{PL}$_{\mathcal{D}}$(\textsf{PRA}) = \textbf{GL.3}}.
\end{theorem}

\begin{proof} By Theorem \ref{theorem:closedFragment}, by Lemma \ref{lemma:upperboundmethod} and by observing that $\Box$ is just $[0]$, it is clear that \textbf{PL}$_{\mathcal{D}}(\textsf{PRA}) \subseteq \textbf{GL.3}$. For the other inclusion, we must show, under the arithmetical interpretation, \[\textsf{PRA} \vdash \Box(\Box A \rightarrow B) \vee \Box (\Box ^+B \rightarrow A),\] for any $A,B \in \mathcal{D}$. However, this follows by arithmetical completeness and by the universality of Ignatiev's frame.

For, suppose $\mathcal{U}, \vec{\alpha} \vDash \Diamond (\Box A \wedge \neg B) \wedge \Diamond (\Box^+B \wedge \neg A)$, for some $\vec{\alpha}$. By Theorem \ref{arithcom} there are root points $\widehat{\beta}$ and $\widehat{\gamma}$, such that $\mathcal{U},\widehat{\beta} \vDash \Box A \wedge \neg B$, and $\mathcal{U},\widehat{\gamma} \vDash \Box^+B \wedge \neg A$. By Lemma \ref{compar}, either $\widehat{\beta} R_0 \widehat{\gamma}$, $\widehat{\gamma}R_0\widehat{\beta}$, or $\widehat{\beta} = \widehat{\gamma}$. All three lead to contradiction.
\end{proof}

\begin{section}{Non-Linear \textbf{GL}-frames}\label{section:NonLinear}

Theorems \ref{theorem:closedFragment} and \ref{theorem:ClosedFragmentGLP} suggest that it may not be straightforward to define a fragment whose associated restricted provability logic is anything other than \textbf{GL.3} or just \textbf{GL}. In this section we fill in this gap by giving some sufficient conditions on constants, so that we obtain logics of non-linear \textbf{GL}-frames. We will be working with generic fragments $\mathcal{F}_n$, with some finite number $n$ of constants: \[\mathcal{F}_n \; := \; s_1 \; \vert \; s_2 \; \vert \; ... \; \vert \; s_n \; \vert \; \bot \; \vert \; \mathcal{F}_n \rightarrow \mathcal{F}_n \; \vert \; \Box \mathcal{F}_n\] As before, we will be viewing formulas in $\mathcal{F}_n$ simultaneously as arithmetical formulas, where each $s_i$ is a specified formula in the language of arithmetic and $\Box$ is the standard provability predicate, and as modal formulas, where each $s_i$ is interpreted as a constant and $\Box$ is a normal modal operator.

\begin{subsection}{Fragments, Logics and Models} Let $\vec{{\bf s_i}}$ stand for the sentence $\bigwedge _{j \in J} s_{j+1} \wedge \bigwedge _{k \in K} \neg s_{k+1}$, where $J$ is the set of places in the binary expansion for $i$ with value 1, and $K$ is the complement of $J$ in $\{  0, \dots,i-1\}$. Then we define the following class of logics.

\begin{definition} 
The logic \textnormal{\textbf{FGL}}$_n$ is formulated in the language $\mathcal{F}_n$ and thus, contains no propositional variables. The axioms and rules are specified by the axioms and rules of \textbf{GL} together with the list of the $2^n$ many axioms below, one axiom for each Boolean combination of the $s_i$. The $B$ in these axioms stands for any formula that is a Boolean combination of formulas of the form $\Box^{\alpha}\bot$, where $\alpha < \omega +1$ and $\Box ^{\omega}\bot := \top$.
\begin{itemize}
\item[] $\Box( \vec{{\bf {s_0}}} \rightarrow B) \rightarrow \Box B$;
\item[] $\vdots$ 
\item[] $\Box( \vec{{\bf s_{2^{n}-1}}} \rightarrow B) \rightarrow \Box B$.
\end{itemize}
\end{definition}

These logics \textnormal{\textbf{FGL}}$_n$ come with an associated model, based on the following frames:
\begin{definition} The frame $\mathfrak{G}_n := \langle G_n, R_n \rangle$, where $G_n := \{\langle m,i \rangle : m \in \omega, i < 2^n\}$, and $\langle m, i \rangle R_n \langle p,j \rangle$ just in case $p < m$.
\end{definition}
The associated model defined on this frame is given \textit{via} the binary expansion, where $J_j$ is given as above, relative to $j$.
\begin{definition} $\mathfrak{G}_n^{\bullet}$ is the triple $\langle G_n, R_n, V_n \rangle$, where $V_n(s_j) = \{\langle m,i \rangle : i \in J_j\}$.
\end{definition}

For a visualization of $\mathfrak{G}_1^{\bullet}$, see Figure \ref{figu:pglsemantics}.

\begin{figure}
\begin{center}
\input{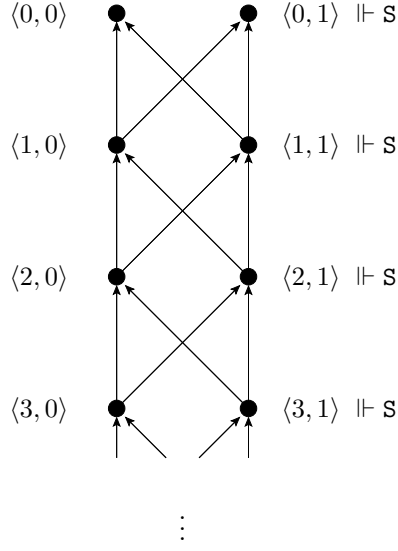}
\caption{The model $\mathfrak{G}_1^{\bullet}$}\label{figu:pglsemantics}
\end{center}
\end{figure}

\begin{theorem} \label{theorem:frodelcompleteness} For all formulas $A \in \mathcal{F}_n$, \textnormal{\textbf{FGL}}$_n \vdash A$, if and only if $\mathfrak{G}_n^{\bullet} \vDash A$.
\end{theorem}
\begin{proof} [Proof Sketch]
The full proof for the case of $\mathcal{F}_1$ is established in \cite{Joo04}. Here we give a sketch for the general case. Soundness is routine. For completeness, we use the following two lemmata.
\begin{lemma} \label{lemma:normalforms} Each $A \in \mathcal{F}_n$ is equivalent in \textnormal{\textbf{FGL}}$_n$ to a Boolean combination of formulas of the form $s_1$,..., $s_n$, or $\Box ^{\alpha}\bot$. In particular \textnormal{\textbf{FGL}}$_n \vdash \Box A \leftrightarrow \Box^{\alpha}\bot$ for some $\alpha < \omega +1$.
\end{lemma}
\begin{lemma} \label{lemmabox} If \textnormal{\textbf{FGL}}$_n \vdash \Box A$, then \textnormal{\textbf{FGL}}$_n \vdash A$.
\end{lemma}
These lemmata are straightforwardly proven by manipulation of modal normal forms. Completeness is then clear. If \textbf{FGL}$_n \nvdash A$, then by Lemma \ref{lemmabox}, \textbf{FGL}$_n \nvDash \Box A$, and by Lemma \ref{lemma:normalforms}, \textbf{FGL}$_n \vdash \Box A \leftrightarrow \Box ^{\alpha} \bot$, for some $\alpha < \omega$ (in particular $\alpha \neq \omega$). By soundness, for any point $\langle m,i \rangle \in G_n$ we know $\mathfrak{G}_n^{\bullet}, \langle m,i \rangle \vDash \Box A \leftrightarrow \Box ^{\alpha} \bot$. Certainly $\mathfrak{G}_n^{\bullet}, \langle \alpha, 0 \rangle \nvDash \Box ^{\alpha} \bot$, so $\mathfrak{G}_n^{\bullet}, \langle \alpha, 0 \rangle \nvDash \Box A$. That, in turn, means for some $\langle \beta, j \rangle$ with $\beta < \alpha$, we have $\mathfrak{G}_n^{\bullet}, \langle \beta , j \rangle \vDash \neg A$. So $A$ is falsified on $\mathfrak{G}_n^{\bullet}$.
\end{proof}
\end{subsection}
\begin{subsection}{Conditions for completeness} \label{subsection:conditions} Suppose we have a given theory $T$ and some fragment $\mathcal{F}_n$, and we would like a characterization of \textbf{PL}$_{\mathcal{F}_n}(T)$. In Section \ref{fragmentslogics} we showed that if condition (1) holds for some logic \textbf{L} and model $\mathcal{M}$, then Theorem \ref{theorem:mainTheorem} will follow. Recall Condition (1):
\[
T \vdash A \ \Leftrightarrow \ \textbf{L} \vdash A \ \Leftrightarrow \ \mathcal{M} \vDash A.
\]
To show (1) holds for this case, one merely needs to show arithmetical soundness and completeness of \textbf{L} for $T$ . However, given Lemmata \ref{lemma:normalforms} and \ref{lemmabox}, arithmetical completeness of \textbf{L} depends only on arithmetical soundness of \textbf{L}. 

To see this, suppose \textbf{FGL}$_n \nvDash A$. Then by Lemma \ref{lemma:normalforms}, \textbf{FGL}$_n \nvDash \Box A$. Since \textbf{FGL}$_n \vDash \Box A \leftrightarrow \Box ^{\alpha} \bot$, for $\alpha \neq \omega$, as long as we have soundness of \textbf{L}, $T \vdash \Box A \leftrightarrow \Box ^{\alpha} \bot$, under the arithmetical interpretation. Now, if moreover $T$ is a sound theory in the sense that it does not prove any false statements we get $T \nvdash \Box A$, from which it follows $T \nvdash A$.

Consequently, the following is a corollary of Theorem \ref{theorem:mainTheorem} and Theorem \ref{theorem:frodelcompleteness}. Note that both image finiteness and definability of the states in the model $\mathfrak{G}_n^{\bullet}$ are evident.

\begin{corollary} \label{corollary:sufficient} \textnormal{\textbf{PL}}$_{\mathcal{F}_n}(T) = \mathcal{L}(\mathfrak{G}_n)$ whenever $[\textnormal{\textbf{FGL}}_n \vdash A \  \Rightarrow T \vdash A.]$
\end{corollary}
In Section \ref{qgl}, we shall see that each of these frames $\mathfrak{G}_n$ has a simple axiomatization. For the rest of this section, we exhibit a suitable constant for the case of $\mathcal{F}_1$.
\end{subsection}
\begin{subsection}{A Constant for \textsf{I$\Sigma _1$}.}
Recall \textsf{I$\Sigma _1$} is the theory \textsf{Q} (\cite{Tarski}) along with induction over $\Sigma _1$ formulas. This theory is finitely axiomatizable, so let $\sigma$ stand for the sentence axiomatizing it. We then define the fragment $\mathcal{Q}$ as a special case of $\mathcal{F}_1$:
\[\mathcal{Q} \; := \; \sigma \; \vert \; \bot \; \vert \; \mathcal{Q} \rightarrow \mathcal{Q} \; \vert \; \Box Q\]
Our theory $T$ will be Primitive Recursive Arithmetic (\textsf{PRA}), essentially just \textsf{Q} with function symbols for all of the primitive recursive functions and induction over $\Delta _0$ formulas. The relationship between \textsf{I$\Sigma _1$} and \textsf{PRA} is well studied and understood (\cite{Min72}, \cite{Par72}, \cite{Bek96}). By Corollary \ref{corollary:sufficient}, we need to show that \textbf{FGL}$_n$ is sound with respect to \textsf{PRA}. It is already well known that \textbf{PL}(\textsf{PRA}) = \textbf{GL}, so certainly all the axioms and rules of \textbf{GL} are sound. We need only observe the following also hold:
\begin{itemize}
\item[(i)] \textsf{PRA} $\vdash \Box(\sigma \rightarrow B) \rightarrow \Box B$,
\item[(ii)] \textsf{PRA} $\vdash \Box(\neg \sigma \rightarrow B) \rightarrow \Box B$.
\end{itemize}
In fact, item (i) is a direct consequence of what is known as Parson's Theorem (named after Charles Parsons, but discovered independently by Grigori Mints and Gaisi Takeuti), which says that \textsf{I}$\Sigma _1$ is $\Pi _2$-conservative over \textsf{PRA}. In \cite{Bek96} it is shown that this theorem is in fact formalizable in \textsf{PRA}, which gives us (i).
\begin{theorem} [Parson's Theorem] \textnormal{\textsf{PRA}} $\vdash \forall ^{\Pi _2} B \  (\Box  (\sigma \rightarrow B) \rightarrow \Box B)$.
\end{theorem}
So this certainly holds for $\mathcal{B}(\Sigma _1)$ formulas consisting of Boolean combinations of formulas of the form $\Box ^{\alpha} \bot$. As for (ii), it is shown in \cite{Joo04} that the negation of the sentence axiomatizing \textsf{I}$\Sigma _1$ is $\Pi _3$-conservative over \textsf{PRA}. That is, we have the following lemma:
\begin{lemma} 
\textnormal{\textsf{PRA}} $\vdash \forall ^{\Pi _3} B \ (\Box (\neg \sigma \rightarrow B) \rightarrow \Box B)$.
\end{lemma}
Thus, we can state the following corollary:
\begin{corollary} \label{corollary:logicofq} \textnormal{\textbf{PL}}$_{\mathcal{Q}}(\textnormal{\textsf{PRA}}) = \mathcal{L}(\mathfrak{G}_1)$.
\end{corollary}
While the logic \textbf{GL.3} of the linear frame $\mathfrak{G}_0$ is well known, that of $\mathfrak{G}_1$ is not. Therefore in the following section we provide a simple axiomatization. Our work can then be generalized to arbitrary $\mathfrak{G}_n$.
\end{subsection}
\end{section}

\begin{section}{The Logic of $\mathfrak{G}_1$} \label{qgl}
\begin{subsection}{The Modal Logic \textbf{GL.4} and its corresponding class of frames}
We define \textbf{GL.4} to be the normal modal logic obtained by adding to \textbf{GL} the following two axiom schemata: \begin{itemize}
\item[\textit{Q1}.] $\Box (\Box A \rightarrow (B \vee C)) \vee \Box (\Box ^+ B \rightarrow (A \vee C)) \vee \Box (\Box ^+ C \rightarrow (A \vee B))$;
\item[\textit{Q2}.] $\Diamond (\Diamond A \wedge \Box B) \rightarrow \Box(\Diamond A \vee B)$.
\end{itemize}

\textbf{GL.4} in fact defines a natural class of frames. We define $\mathcal{C}$ to be the class satisfying the following properties: \begin{itemize}\item[\textit{C1}.] Finite, irreflexive and transitive; \item[\textit{C2}.] Non-triple branching: $(xRy \; \& \; xRz \; \& \; xRw) \Rightarrow$ \[(wRy \vee yRw \vee zRw \vee wRz  \vee yRz \vee zRy \vee w=y  \vee z=y  \vee w=z);\]
\item[\textit{C3}.] Strongly confluent: $(xRy \; \& \; xRz \; \& \; yRw) \Rightarrow (zRw \vee wRz \vee yRz)$.
\end{itemize}

\begin{theorem} \label{qglcom} \textnormal{\textbf{GL.4}} is sound and complete with respect to $\mathcal{C}$.
\end{theorem}

Soundness is proven as usual by induction on complexity of proofs. As for completeness, we shall appeal to the canonical model of \textbf{GL.4} (see Definition 4.18 of \cite{Bla01}). In particular we use the finite filtration method to transform the canonical model into a model in the class $\mathcal{C}$.

Recall the canonical model $\mathfrak{M}$ of \textbf{GL.4} is the triple, $\langle W^{GL.4}, R^{GL.4}, V^{GL.4} \rangle$ with \begin{itemize} \item $W^{GL.4}$ is the set of maximal \textbf{GL.4}-consistent sets;
\item For $\Gamma, \Delta \in W^{GL.4}$, define $\Gamma R^{GL.4} \Delta$ if for all $\phi \in \Delta$ we have $\Diamond \phi \in \Gamma$; \item $V(p) = \{\Gamma : p \in \Gamma\}$, for propositional variables $p$. \end{itemize}
First, we make some key observations about this model, the verifications of which are straightforward.
\begin{lemma} \label{c2} C2 holds on $\mathfrak{M}$. \end{lemma}
\begin{lemma} \label{c3} C3 holds on $\mathfrak{M}$. \end{lemma}
In fact, these follow by the fact that axiom \textit{Q1} is \textit{canonical} for property \textit{C2}, as is axiom \textit{Q2} for \textit{C3} (see \cite{Bla01}, Definition 4.31). Thus, it remains to show that we can transform the underlying frame of $\mathfrak{M}$ into a finite partial order, while preserving validity of formulas.

\begin{proof} [Proof of Theorem \ref{qglcom}]
Suppose that \textbf{GL.4} $\nvdash A$, for some formula $A$. We would like to find a maximal consistent set $\Gamma$ such that $(\Box A \wedge \neg A) \in \Gamma$, so that $\Gamma$ is an `irreflexive' point in the canonical model.

By the fact that $A$ is not a theorem, we are guaranteed of some $\Delta \in W^{GL.4}$ such that $A \notin \Delta$. If $\Box A \in \Delta$, then set $\Gamma := \Delta$. Otherwise, since $\neg \Box A \in \Delta$, by the contrapositive form of L\"{o}b's Theorem $\Diamond (\Box A \wedge \neg A) \in \Delta$. Thus by the so-called `Existence Lemma' (\cite{Bla01}, Lemma 4.20) for normal modal logics, $\Delta$ is $R^{GL.4}$-related to some $\Sigma$ for which $(\Box A \wedge \neg A) \in \Sigma$. In that case, set $\Gamma := \Sigma$.

Either way we have some $\Gamma$ with $(\Box A \wedge \neg A) \in \Gamma$. Notice also, if $\Box C$ is a subformula of $A$, and $\Box C \notin \Gamma$, then by the same argument there is some `irreflexive' $\Delta$ such that $\Gamma R^{GL.4} \Delta$ and that $(\Box C \wedge \neg C) \in \Delta$. Moreover, by Lemma \ref{c2} there are at most two distinct such $\Delta$.

With these observations in place, our filtrated model $\mathfrak{M'} = \langle W,R,V \rangle$ will be defined as a submodel of $\mathcal{M}$: \begin{itemize} \item[(i)] $W := \{\Gamma\} \cup \{\Delta: \Gamma R^{GL.4} \Delta$, and there is $\Box C$ subsentence of $A$, such that \\ $(\Box C \wedge \neg C) \in \Delta$ and $\neg \Box C \in \Gamma\}$;
\item[(ii)] $R$ is just $R^{GL.4}$ restricted to points in $W$;
\item[(iii)] $V(p) := V^{GL.4}(p) \cap W$.
\end{itemize}
The model $\mathfrak{M'}$ satisfies \textit{C2}, \textit{C3}, and transitivity simply because $\mathfrak{M}$ does. It is clearly finite. And irreflexivity, as hinted above, follows from the fact that each point in $W$ was chosen to contain some formulas $\Box C$ and $\neg C$, ensuring the point is not related to itself. It follows $\mathcal{M'}$  is in $\mathcal{C}$.

The standard `Truth Lemma' is then proven by induction:
\begin{lemma} If $\Delta \in W$ and $B$ is a subsentence of $A$, then $B \in \Delta$ iff $\mathfrak{M'}, \Delta \vDash B$. \end{lemma}
Concluding the proof, since $A \notin \Gamma$, we have that $\mathfrak{M'},\Gamma \nvDash A$.
\end{proof}
\end{subsection}

\begin{subsection}{The Class $\mathcal{C}$ and the Frame $\mathfrak{G}_1$}

We must now show that \textbf{GL.4} is the logic of the frame $\mathfrak{G}_1$. 

Recall a $p$-morphism from $\mathbb{F} = \langle W,R \rangle$ to $\mathbb{F'} = \langle W',R' \rangle$ is a function $f: W \rightarrow W'$, such that $xRy$ implies $f(x)R'f(y)$; and if $f(x)R'y'$ then there is some $y \in W$ such that $f(y) = y'$ and $xRy$. The following theorem is standard:\footnote{See, e.g. \cite{Bla01}, Definition 3.13, where $p$-morphisms go under the name \textit{bounded morphism}.}
\begin{theorem} If there is a $p$-morphism from $\mathbb{F}$ to $\mathbb{F'}$, then the existence of a valuation $V'$ and point $w' \in W'$ such that $\langle \mathbb{F'},V' \rangle , w' \nvDash A$, ensures the existence of a valuation $V$ and point $w \in W$, such that $\langle \mathbb{F}, V \rangle , w \nvDash A$.
\end{theorem}

To demonstrate that \textbf{GL.4} is the logic of $\mathfrak{G}_1$, we use the following proposition:
\begin{pro} \label{pmorph} For any frame $\mathbb{F} \in \mathcal{C}$ and any point $x$ in $\mathbb{F}$, there is some point $\langle m,i \rangle$ in $\mathfrak{G}_1$, such that there exists a $p$-morphism from the subframe generated by $\langle m,i \rangle$ to the subframe generated by $x$.
\end{pro}

In other words, falsifiability is reflected by $p$-morphisms, which gives us the following corollary of Proposition \ref{pmorph} and improvement upon Corollary \ref{corollary:logicofq}.
\begin{corol} \textnormal{\textbf{PL}}$_\mathcal{Q}$\textnormal{(\textsf{PRA}) =} \textnormal{\textbf{GL.4}}.
\end{corol}
It remains only to verify Proposition \ref{pmorph}.

\begin{proof} [Proof Sketch of Proposition \ref{pmorph}]
The proof proceeds by induction on the number of points in a frame in $\mathcal{C}$. The basic case is obvious. Supposing we have a frame with one point, say $x$, then consider the subframe generated by $\langle 0,0 \rangle$, and the $p$-morphism mapping $\langle 0,0 \rangle$ to $x$.

Supposing we have a frame in $C$ with $n+1$ points, consider the subframe $\mathbb{F} = \langle W,R \rangle$ generated by some point $x\in C$. We would like to use the inductive hypothesis to obtain a $p$-morphism to some subframe of $\mathbb{F}$ containing $\leq n$ points, and extend it to all of $\mathbb{F}$. To do this we consider three cases: (i) $x$ has no successors; (ii) $x$ has one immediate successor (i.e. point $y$ such that $xRy$ and there is no $z$ with $xRzRy$); and (iii) $x$ has two immediate successors. More than 2 immediate successors is ruled out by property \textit{C2}.

Case (i) is trivial. For case (ii), let $\mathbb{F'}$ be $\mathbb{F}$ without the point $x$, and let $y$ be the unique immediate successor of $x$. Then since $\mathbb{F'} \in \mathcal{C}$ and it has $n$ points, we have a $p$-morphism $f$ from the subframe generated by some point $\langle m,i \rangle$ in $\mathfrak{G}_1$ to $\mathbb{F'}$, the subframe generated by $y$. We then consider the subframe generated by $\langle m+1,i \rangle$ instead, and extend the $p$-morphism $f$ so that $f(\langle m+1,i \rangle) = x$ and $f(\langle m,i-1 \rangle) = y$.

Verifying case (iii) is similar, except that instead of removing the point $x$, we must remove the `maximal' points of $\mathbb{F}$. Then the $p$-morphism obtained by inductive hypothesis is extended by shifting each point in the morphism by one. Thus, e.g. if $\langle m,i \rangle$ is mapped to $y$, then in the new mapping $\langle m+1,i \rangle$ is mapped to $y$. And we let $f(\langle 0,0 \rangle) = f(\langle 0,1 \rangle) = x$. The details are straightforward and are left to the reader (or can be found in \cite{Ica07}).
\end{proof}

\end{subsection}

\begin{remark} The methods in this section carry over to the general case of frames $\mathfrak{G}_n$ for arbitrary $n$. By an analogous argument, one can prove the logic is simply \textit{Q2} (strong confluence) and the axiom corresponding to ``non-n+2-ary-branching", which is just a generalization\footnote{It is not hard to see that $\Box(\Box A \to B) \vee \Box(\Box^+ B \to A)$ is equivalent to $\Box(\Box^+ A \to B) \vee \Box(\Box^+ B \to A)$ over \textbf{GL}. } of non-branching and non-triple-branching:
\[\bigvee_{i \leq n+1} \Box(\Box^+ A_i \to \bigvee_{i\neq j} A_j).\]
\end{remark}

\end{section}

\begin{section}{On the proof of Solovay's Theorem} \label{section:Solovay}
In Sections \ref{section:closedfragment} and \ref{section:glp} we showed that \textbf{PL}$_{\mathcal{F}}(T) =$ \textbf{GL.3} for a wide range of arithmetical theories $T$ and fragments $\mathcal{F}$. Otherwise put, \textbf{PL}$_{\mathcal{F}}(T)$ gives us the logic of non-branching \textbf{GL}-frames. \textit{Prima facie}, one might imagine the possibility of strategically adding sentences into the fragment $\mathcal{F}$ (where $\mathcal{F}$ is, e.g. $\mathcal{B}$), so as to obtain the logic of non-triple-branching \textbf{GL}-frames, then that of non-quadruple-branching \textbf{GL}-frames, and so on. Assuming this could be generalized it would be possible to define an infinite fragment $\mathcal{H}$, for which \textbf{PL}$_{\mathcal{H}}(T) =$ \textbf{GL}. At that point, to the extent that Solovay's Theorem is not already assumed in the determination of $\mathcal{H}$, we would have a new proof of the result. After all, any non-theorem of \textbf{GL} can be falsified on some finite, and thus finitely branching, frame. So the witnessing realization would make use of some finite subset of the fragment, sufficient to falsify the formula.

What we have shown is that the first step in this process is (almost) possible, \textit{vis-\`{a}-vis} Corollary \ref{corollary:logicofq}. Adding the constant for \textsf{I$\Sigma _1$} and capitalizing on the well studied relationship between that theory and \textsf{PRA}, we are able to obtain the logic of non-triple-branching (and strongly confluent) \textbf{GL}-frames. Two important questions remain, however, before taking the next step.

The first and most obvious question is what the further constants will be. The particular case of \textsf{I}$\Sigma _1$ and \textsf{PRA} is already well studied. Going beyond that may require some significant arithmetical investigation. In Section \ref{subsection:conditions} we isolated what arithmetical facts are sufficient to hold. So on the proposed strategy it would simply be a matter of finding a theory and a fragment that satisfy these requirements.

The second, and more curious, question is how to dispense with property \textit{C3}, strong confluence. We have seen that the logic of the frame \textbf{G}$_n$ always contains the formula \textit{Q2}, and so it will clearly remain in the limit. However \textit{Q2} is obviously not a theorem of \textbf{GL}. Finding constants whose associated provability logics do not validate \textit{Q2} may prove a challenge. Understanding this situation may shed light on Solovay's original proof.

\end{section}

\section{Interpretability Logics with Restricted Substitutions}\label{section:InterpretLogics}

Interpretations are used throughout mathematics and logic.
Loosely speaking, an interpretation from a theory $V$ into a theory $U$ is structure preserving map that translates theorems of $V$ to theorems of $U$. The notion of interpretability that we discuss below is \emph{grosso modo} that of \cite{Tarski} and details can be found in, e.g.\ \cite{JapJongh} or in \cite{Visser97}.   

\subsection{Interpretability Logics}
Interpretability can be seen as a generalization of provability. By $\alpha \rhd_T \beta$ we denote a natural formalized version of the statement that $T + \beta$ is interpretable in $T+\alpha$.

\emph{Interpretability Logics} are designed to capture the structural behavior of formalized interpretability. The language of these logics is that of provability logic together with a binary modality $\rhd$, orthographically identical to the arithmetical operator, to model formalized interpretability. And indeed, \emph{arithmetical realizations} are extended as expected by imposing that 
\[
(A \rhd B)^* \ =\  A^* \rhd B^*. 
\]  
For a clear distinction, let \formil denote the class of modal formulas in language of interpretability logic and \formgl the standard modal language of basic provability logic. In analogy with the definition of \textbf{PL}$(T)$ we define \textbf{IL}$(T)$, the interpretability logic of a theory $T$
\[
\textbf{IL}(T) := \{  A \in \formil \mid \forall * \ T \vdash A^*\}  \ \ \ \mbox{and}  \ \ \ \textbf{IL}_{\Gamma}(T) := \{  A \in \formil \mid \forall *{\in} \Gamma \ T \vdash A^*\}.
\]
By Theorem \ref{theorem:solovay} and Footnote \ref{footnote:ClassOfTheories} we see that provability logics are the same for all sufficiently strong theories. This is certainly not the case for interpretability logics, which turn out to be more sensitive to differences between theories. One such example is the notion of an \emph{essential reflexive} theory. 

A theory is \emph{reflexive} if it proves the consistency of any finite subpart of it. A theory is essentially reflexive whenever any finite extension of it is reflexive. The following theorem is due independently to A.\ Berarducci and V. Shavrukov. The definition of \ilm will follow below.
\begin{theorem}[Berarducci \cite{bera:inte90},  Shavrukov \cite{shav:logi88}]\label{theo:shav}
If $T$ is an essentially reflexive and $\Sigma_1$ sound theory, then $\intl{T}=\ilm$.
\end{theorem}
However, if a theory is finitely axiomatizable we get a different outcome where, again, \ilp is defined below.
\begin{theorem}[Visser \cite{viss:inte90}] \label{thm:ilp}
If $T$ is finitely axiomatizable, $\Sigma_1$ sound, and extending  $\textsf{I}\Delta_0 + {\sf supexp}$, then $\intl{T}=\ilp$.
\end{theorem}

A prominent problem in formalized interpretability is to determine the maximal interpretability logic that is contained in any reasonable arithmetical theory. 

\begin{definition}\label{defi:ilall}
The interpretability logic of all reasonable arithmetical theories, 
written \intl{All}, is the set of formulas $\varphi$ such that for all $T$ and $*$,  $T\vdash \varphi^*$. Here we let $T$ range over all reasonable\footnote{The boundaries are not exactly determined and will depend a bit on the answer. It is legitimate to think of any theory extending $\textsf{I}\Delta_0 + \exp$.} arithmetical theories. 
\end{definition}

Clearly, \intl{All} is in the intersection of \ilm and \ilp but apparently it possesses a very rich structure (see \cite{joo:prol00}, and \cite{GorisJoostenNewPrinc}). In this paper, it is only important to know that a certain very weak logic to be defined below is part of \intl{\pra}.
\begin{fact}\label{fact:IlwInILall}
$\ilw \subset \intl{\pra}$
\end{fact}
For most theories that do not fall under Theorems \ref{theo:shav} and \ref{thm:ilp}, the interpretability logic is  unknown. The theory \pra is a notable example: the logic \intl{\pra} is still unknown. The most recent results for \intl{\pra} are presented in \cite{BilDickJo}. 

\pra is known to be the same as \ir{1} where \ir{n} is defined as $\textsf{I}\Delta_0 + \exp$ plus the $\Sigma_n$ induction rule. See for example \cite{Bek:ReflArith}. In that paper a proof can also be found for the following theorem.

\begin{theorem}\label{theorem:ReflexiveExtensionsOfRules}
 \ir{n} is reflexive, as is any extension of \ir{n} by $\Sigma_{n+1}$ formulas.
\end{theorem}

The logical complexity of interpretability is $\Sigma_3$ and in \cite{Shav97} it is shown that it is essentially so. However, by a theorem due to Orey and H\'ajek we can often reduce the $\Sigma_3$ notion of interpretability to the $\Pi_2$ notion of $\Pi_1$-conservativity. A theory $V$ is $\Pi_1$-conservative over $U$, we write $U\rhd_{\Pi_1}V$, whenever for all $\Pi_1$ sentences $\pi$ we have that [$V\vdash \pi$ implies $U\vdash \pi$]. 

\begin{theorem}[Orey-H\'ajek]\label{theorem:OH}
For reflexive theories $U$ and $V$ we have 
\[
(U\rhd V ) \ \ \ \Leftrightarrow \ \ \ (U \rhd_{\Pi_1}V)
\]
and this equivalence is provable in \ea.
\end{theorem}

One advantage of this characterization is evidently that the logical complexity of $\Pi_1$-conservativity is lower than that of interpretability. Another advantage is that the so-called $\Pi_1$-conservativity logic is a relatively stable notion. The $\Pi_1$-conservativity logic of a theory $T$ is just the set of modal formulas in \formil that are provable in $T$ under any arithmetical realization where the $\rhd$ modality is mapped to $\rhd_{\Pi_1}$.

\begin{theorem}\label{theorem:PiILM}
For any sound theory $T$ extending $\textnormal{\textsf{I}}\Pi_1^-$ we have that the $\Pi_1$-conservativity logic of $T$ is \ilm.
\end{theorem}

The theorem was first proven by H\'ajek and Montagna in \cite{HM90} and \cite{HM92} to hold for any sound theory containing $\textsf{I}\Sigma_1$. Beklemishev and Visser in \cite{BeklVisser2005} lowered the threshold to the rather weak theory $\textsf{I}\Pi_1^-$ that allows only induction for parameter free formulas of complexity $\Pi_1$. It is well known that \pra extends $\textsf{I}\Pi_1^-$ (\cite{Bek:ReflArith}). 

\begin{remark}\label{remark:Sigma2Subs}
The proof of Theorem \ref{theorem:PiILM} is rather similar to that of Solovay's original proof and again (see Theorem \ref{theorem:BoolSigmaPL}), the substitutions in the completeness proof can be taken\footnote{Albert Visser (p.c.) notes that close inspection of the proof actually reveals that the substitutions can be taken to be $\Delta_2(\textsf{I}\Pi_1^-)$. That is, a $\Sigma_2$ sentences that is probably in $\textsf{I}\Pi_1^-$ equivalent to a $\Pi_2$ sentence.} to be $\Sigma_2$.
\end{remark}

The logics \ilm and \ilp have elegant syntactical presentations. We shall define them in parts. First, we define a logic \il that is present to all interpretability logics studied. Next this logic \il is extended by adding more axiom schemata.

(When we write formulas in \formil we adhere to the following binding conventions. We say that $\rhd$ binds stronger than $\to$ but weaker than all other connectives. Using this convention we can save a lot of brackets.)

\begin{definition}\label{defi:il}
The logic \il is the smallest set of formulas being closed under
the rules of Necessitation and of Modus Ponens, that contains
all tautological formulas and all instantiations of the following
axiom schemata.

\begin{enumerate}
\item[${\sf L1}$]\label{ilax:l1}
        $\Box(A\rightarrow B)\rightarrow(\Box A\rightarrow\Box B)$
\item[${\sf L2}$]\label{ilax:l2}
        $\Box A\rightarrow \Box\Box A$
\item[${\sf L3}$]\label{ilax:l3}
        $\Box(\Box A\rightarrow A)\rightarrow\Box A$
\item[${\sf J1}$]\label{ilax:j1}
        $\Box(A\rightarrow B)\rightarrow A\rhd B$
\item[${\sf J2}$]\label{ilax:j2}
        $(A\rhd B)\wedge (B\rhd C)\rightarrow A\rhd C$
\item[${\sf J3}$]\label{ilax:j3}
        $(A\rhd C)\wedge (B\rhd C)\rightarrow A\vee B\rhd C$
\item[${\sf J4}$]\label{ilax:j4}
        $A\rhd B\rightarrow(\Diamond A\rightarrow \Diamond B)$
\item[${\sf J5}$]\label{ilax:j5}
        $\Diamond A\rhd A$
\end{enumerate}
\end{definition}

Apart from the axiom schemata enumerated in Definition \ref{defi:il} we will
need consider other axiom schemata too.

\begin{enumerate}
\item[${\sf M}$] 
$A \rhd B \rightarrow A \wedge \Box C \rhd B \wedge \Box C$

\item[${\sf P}$]
$A \rhd B \rightarrow \Box (A \rhd B)$


\item[${\sf W}$]
$A \rhd B \rightarrow A \rhd B \wedge \Box \neg A$
%
%


\end{enumerate}

If $\sf X$ is a set of axiom schemata we will denote by \extil{X} the
logic that arises by adding the axiom schemata in $\sf X$ to \il.

\subsection{The closed fragment}
Because closed formulas in \ilw can be reduced to those of \gl (\cite{haje:norm91}) we can prove that $\restintl{\pra}{\mathcal{B}}$ is again the logic of linear frames.

\begin{definition} 
The logic $\mbox{\bf ILW.3}$ is obtained by adding 
the linearity axiom schema $\Box (\Box A \rightarrow B) \vee
\Box (\boxdot B \rightarrow A)$ to \ilw. 
\end{definition}

\begin{theorem}\label{ril}
$\restintl{\pra}{\mathcal{B}}=\mbox{\bf ILW.3} $
\end{theorem}

\begin{proof}
We give a translation from formulas $\varphi$ 
in \formil
to formulas $\varphi^{\sf tr}$ in 
\formgl
such that 
\[
\begin{array}{cl}
\mbox{\bf ILW.3}\vdash \varphi \Leftrightarrow \mbox{\bf GL.3} \vdash 
\varphi^{\sf tr} & (*)\\ 
\mbox{ and } & \ \\
\mbox{\bf ILW.3} \vdash \varphi \leftrightarrow \varphi^{\sf tr}. & (**) 
\end{array}
\]
If we moreover know $(*{*}*): \  \  \mbox{\bf ILW.3} \vdash \varphi \Rightarrow 
\forall \, *{\in} \mathcal{B}\; \pra \vdash \varphi^*$
we would be done. For then we have by $(**)$ and $(*{*}*)$ that 
\[
\begin{array}{ll}
\forall \, * {\in} {\sf Sub}(\mathcal{B})\; \pra \vdash
\varphi^* \leftrightarrow (\varphi^{\sf tr})^*
\end{array}
\]
\begin{center}
and consequently
\end{center}
\[
\begin{array}{ll}
\forall \, * {\in} \mathcal{B}\; \pra \vdash \varphi^* & \Leftrightarrow \\
\forall \, * {\in} \mathcal{B}\; \pra \vdash {(\varphi^{\sf tr})}^* &
\Leftrightarrow \\
\mbox{\bf GL.3} \vdash \varphi^{\sf tr} & \Leftrightarrow \\ \mbox{\bf ILW.3}\vdash \varphi.
\end{array}
\]

We first see that $(*{*}*)$ holds. 
Certainly, by Fact \ref{fact:IlwInILall}, we have that 
$\ilw \subseteq \restintl{\pra}{\mathcal{B}}$. Thus it remains to show that 
$\pra \vdash \Box (\Box A^* \rightarrow B^*) \vee \Box (\boxdot B^* \rightarrow A^*)$ 
for any formulas $A$ and $B$ in \formil and any $*{\in}\mathcal{B}$.
As any formula in the closed fragment of \ilw is 
equivalent to a formula in the closed fragment of \gl (see \cite{haje:norm91}), Theorem 
\ref{theorem:closedFragment} gives us that indeed the linearity axiom holds for the 
closed fragment of \gl.

Our translation will be the identity translation except for $\rhd$. In that
case we define
\[
(A\rhd B)^{\sf tr} := \Box (A^{\sf tr} \rightarrow (B^{\sf tr}\vee 
\Diamond B^{\sf tr})).
\]

We first see that we have $(**)$. It is sufficient to show that
$\mbox{\bf ILW.3}\vdash p\rhd q \rightarrow \Box (p \rightarrow (q\vee \Diamond q))$.
We reason in $\mbox{\bf ILW.3}$.
An instantiation of the linearity axiom gives us 
$\Box (\Box \neg q \rightarrow (\neg p \vee q)) \vee 
\Box ((\neg p \vee q) \wedge \Box (\neg p \vee q)\rightarrow \neg q)$. The
first disjunct 
immediately yields
$ \Box (p \rightarrow (q\vee \Diamond q))$.

In case of the second 
disjunct we get by propositional logic
$\Box (q \rightarrow \Diamond (p\wedge \neg q))$ and thus 
also 
$\Box (q \rightarrow \Diamond p )$. Now we assume $p\rhd q$. By ${\sf W}$ we 
get $p\rhd q\wedge \Box \neg p$. Together 
with $\Box (q \rightarrow \Diamond p)$, this gives us
$p\rhd \bot$, that is $\Box \neg p$. Consequently we have
$ \Box (p \rightarrow (q\vee \Diamond q))$.

We now prove $(*)$.
By induction on $\mbox{\bf ILW.3}\vdash \varphi$ we see that 
$\mbox{\bf GL.3} \vdash \varphi^{\sf tr}$. All the specific interpretability 
axioms turn out to be provable under our translation in $\gl$.
The only axioms where the $\Box A \rightarrow \Box \Box A$ axiom scheme 
is really used is in ${\sf J_2}$ and  ${\sf J_4}$. To prove the 
translation of  ${\sf W}$ we also need ${\sf L_3}$.

If $\mbox{\bf GL.3} \vdash \varphi^{\sf tr}$ then certainly 
$\mbox{\bf ILW.3} \vdash \varphi^{\sf tr}$ and by $(**)$, 
$\mbox{\bf ILW.3} \vdash \varphi$.
\end{proof}

We thus see that $\mbox{\bf ILW.3}$ is an upperbound for $\intl{\pra}$. Using the
translation from the proof of Theorem \ref{ril}, it is not hard to see that both the
principles ${\sf P}$ and ${\sf M}$ are provable in $\mbox{\bf ILW.3}$. This tells us
that the upperbound is actually not very informative as we know that 
$\intl{\pra}\nvdash \principle{M}$.
By a straight-forward generalization of Lemma \ref{lemma:upperboundmethod} we see that choosing larger
$\Gamma$ will generally yield a smaller $\restintl{\pra}{\Gamma}$ and thus 
a sharper upperbound. Subsection \ref{subsection:refutingM} consists of reflections on just how large the $\Gamma$ should be as to refute \principle{M} in $\restintl{\pra}{\Gamma}$. First we shall include some observations on a fragment slightly larger than the closed fragment.

\subsection{The closed fragment with a constant for \isig{1}}

If we consider the proof of Theorem \ref{theorem:mainTheorem}, we see that it does not make any assumptions on the signature of the modal logic under considerations. In particular, the theorem still holds for interpretability logics. In the theorem below we use this to give a semantic characterization of $\restintl{\pra}{\mathcal{F}_1}$.

In \cite{Joo05} it is established that for a certain frame, that we will denote here by $\widetilde{\mathfrak{G}_1^{\bullet}}$, we have the following equivalence.
\[
\forall A \in \mathcal{F}_1\ [\ \widetilde{\mathfrak{G}_1^{\bullet}}\models A \ \ \Leftrightarrow \ \  \pra \vdash A\ ] \ \ \ \ \ \ (\dag)
\]
For the purpose of this paper it is not material to know what exactly the frame $\widetilde{\mathfrak{G}_1^{\bullet}}$ looks like and we shall refrain from giving a formal definition. It is only important to know that $\widetilde{\mathfrak{G}_1^{\bullet}}$ is just $\mathfrak{G}_1^{\bullet}$ with some additional accessibility relations to model the $\rhd$ modality. This, together with the mere equivalence $(\dag)$ suffices to obtain the following theorem.

\begin{theorem}
$\restintl{\pra}{\mathcal{F}_1} = \mathcal{L}(\widetilde{\mathfrak{G}_1^{\bullet}})$
\end{theorem}

\begin{proof}
Image-finiteness and definability of separate points is clear as interpretability logic is an extension of provability logic. Thus, by Theorem \ref{theorem:mainTheorem} we obtain the result.
\end{proof}

In \cite{Joo05}, also a logic \textbf{PIL} is given such that we actually have
\[
\forall A \in \mathcal{F}_1\ [\ \widetilde{\mathfrak{G}_1^{\bullet}}\models A \ \ \Leftrightarrow \ \  \pra \vdash A\ \ \Leftrightarrow \ \ \mathbf{PIL}\vdash A\ ] .\ \ \ \
\]
This suggests that the following conjecture should not be too hard to prove. In this conjecture, \textbf{ILM.4} denotes the logic that arises by joining \ilm and \textbf{GL.4}.
\begin{conjecture}
$\mathcal{L}(\widetilde{\mathfrak{G}_1^{\bullet}})= \ilm\textbf{.4}$ 
\end{conjecture}
The inclusion $\mathcal{L}(\widetilde{\mathfrak{G}_1^{\bullet}})\supseteq \ilm\textbf{.4}$ is actually very easy and follows from a direct verification of the validity of the axioms on $\widetilde{\mathfrak{G}_1^{\bullet}}$. The other direction is harder but not too interesting as we still have $\principle{M} \in \restintl{\pra}{\mathcal{F}_1}$.

\subsection{Fragments for refuting \principle{M} in $\restintl{\pra}{\Gamma}$}\label{subsection:refutingM}

In \cite{Visser97} it is shown that 
$\intl{\pra}\nvdash A\rhd \Diamond B \rightarrow \Box (A\rhd \Diamond B)$. It is easy to see that $\ilm \vdash  A\rhd \Diamond B \rightarrow \Box (A\rhd \Diamond B)$. This 
implies that \principle{M} is certainly not derivable in \intl{\pra}. We can also find explicit 
realizations that violate \principle{M}, as the following lemma tells us.

\begin{lemma}\label{lemm:vetgeenM}
For $n\geq 1$, we have that $\intl{\ir{n}}\nvdash \principle{M}$.
\end{lemma} 

\begin{proof}
We define a realization $*$
such that 
$\ir{n} \nvdash (p\rhd q \rightarrow p\wedge \Box r \rhd q \wedge \Box r)^*$.

It is well-known that 
$\ir{n} \subsetneq
\isig{n}\subsetneq \ir{n+1}$ and that, for every  
$n{\geq} 1$, $\isig{n}$ is finitely axiomatized.
Let $\sigma_n$ be the single sentence axiomatizing $\isig{n}$.
It is also known that (for $n\geq 1$) 
$\intl{\isig{n}}=\ilp$ and that
$\ilp \nvdash p\rhd q \rightarrow p\wedge \Box r \rhd q \wedge \Box r$.
Thus, for any $n{\geq} 1$ we can find $\alpha_n ,\beta_n$ and $\gamma_n$
such that 
\[
\isig{n}\nvdash \alpha_n\rhd\beta_n\rightarrow 
\alpha_n\wedge \Box \gamma_n \rhd \beta_n \wedge \Box \gamma_n.
\]
Note that 
\[
\ea \vdash \alpha_n\rhd_{\isig{n}}\beta_n \leftrightarrow 
\sigma_n \wedge \alpha_n\rhd_{\ir{n}} \sigma_n \wedge \beta_n
\]
and 
\[
\ea \vdash \Box_{\isig{n}}\gamma_n \leftrightarrow 
\Box_{\ir{n}}(\sigma_n \rightarrow \gamma_n).
\]
Thus, we
have
\[
\ir{n}\nvdash
\sigma_n \wedge \alpha_n\rhd \sigma_n \wedge \beta_n \rightarrow
\sigma_n \wedge \alpha_n \wedge \Box (\sigma_n \rightarrow \gamma_n)
\rhd 
\sigma_n \wedge \beta_n \wedge \Box (\sigma_n \rightarrow \gamma_n)
\]
and we can take $p^*=\sigma_n \wedge \alpha_n$, $q^*=\sigma_n \wedge \beta_n $ 
and $r^*=\sigma_n \rightarrow \gamma_n$.
\end{proof}

We see that the realizations used in the proof of Lemma \ref{lemm:vetgeenM}
get higher and higher complexities. The complexity is certainly 
higher than $\Sigma_2$. 

By Theorem 1 from \cite{BilDickJo} (Theorem 12.1.1 from \cite{Joo04}) we know that for $\alpha,\beta \in \Sigma_2$ we have 
\[
\pra \vdash (\alpha \rhd \beta) \to ((\alpha \wedge \Box \gamma)\rhd (\beta \wedge \Box \gamma))
\]
for any sentence $\gamma$.
This translates to $\restintl{\pra}{\Sigma_2}\vdash \principle{M}$ and indicates that an
arithmetical completeness proof for \intl{\pra} can not work with only 
$\Sigma_2$-realizations.

For 
$\ir{n}$, $n\geq 2$ we know that $\intl{\ir{n}}\subset \ilm$. This 
follows from the next lemma.


\begin{lemma}\label{restrict}
$\restintl{\ir{n}}{\Sigma_2}=\restintl{\ir{n}}{\Delta_{n+1}}=\ilm$
whenever $n\geq 2$.
\end{lemma}

\begin{proof}
We shall use that the logic 
of $\Pi_1$-conservativity for theories containing $\textsf{I}\Pi_1^-$ is 
\ilm as mentioned in Theorem \ref{theorem:PiILM}. 
 
If, for two classes of sentences we have 
$X\subseteq Y$, then $\restintl{T}{Y} \subseteq \restintl{T}{X}$. 
We will thus show that 
$\restintl{\ir{n}}{\Sigma_2}\subseteq \ilm$ 
and
$\ilm \subseteq \restintl{\ir{n}}{\Delta_{n+1}}$.

First, we prove by induction on the complexity of a modal formula $A$
that 
for all $ * {\in} \Delta_{n+1}\ \ \ir{n}\vdash A^*_{\Pi_1}
\leftrightarrow A^*_{\rhd}$ and that the 
logical complexity of $A^*_{\Pi_1}$ is at most $\Delta_{n+1}$.
The basis is trivial and the only interesting induction step is 
whenever $A= (B\rhd C)$. We reason in $\ir{n}$:

\[
\begin{array}{cl}

(B\rhd C)^*_{\rhd}  		&	\leftrightarrow_{\mbox{def.}}\\

\ir{n} + B^*_{\rhd} \rhd \ir{n} + C^*_{\rhd}  
&\leftrightarrow_{\mbox{i.h.}}\\

\ir{n} + B^*_{\Pi_1} \rhd \ir{n} + C^*_{\Pi_1}
&\leftrightarrow_{\mbox{Orey-H\'ajek}}\\

\ir{n} +B^*_{\Pi_1}  \rhd_{\Pi_1} \ir{n} + C^*_{\Pi_1}
&\leftrightarrow_{\mbox{def.}}\\

(B\rhd C)^*_{\Pi_1}& \\

\end{array}
\]

Note that we have access to the Orey-H\'ajek characterization as 
$B^*_{\Pi_1}$ is at most of complexity $\Delta_{n+1}$ and thus 
$\ir{n} +B^*_{\Pi_1}$ is a reflexive theory by Theorem \ref{theorem:ReflexiveExtensionsOfRules}. Also note that 
$(B\rhd C)^*_{\Pi_1}$ is a $\Pi_2$-sentence and thus certainly 
$\Delta_{n+1}$ whenever $n\geq 2$.

If now $\ilm \vdash A$
then 
$\ir{n}\vdash  A^*_{\Pi_1}$
and thus whenever $* \in \Delta_{n+1}$,
$\ir{n}\vdash  A^*_{\rhd}$ and 
$\ilm \subseteq \restintl{\ir{n}}{\Delta_{n+1}}$.

If $\ilm \nvdash A$ then by Remark \ref{remark:Sigma2Subs} for some $* \in \Sigma_2$ we
have $\ir{n}\nvdash  A^*_{\Pi_1}$ whence 
$\ir{n}\nvdash  A^*_{\rhd}$. We may conclude that  
$\restintl{\ir{n}}{\Sigma_2}\subseteq \ilm$.
\end{proof}

\begin{theorem}
$\intl{\pra}\subset \ilm$
\end{theorem}

\begin{proof}
Although the proof of Lemma \ref{restrict} does not give us that $\restintl{\ir{1}}{\Sigma_2}= \ilm$, it does give us that $\restintl{\ir{1}}{\Sigma_2}\subseteq \ilm$. By earlier observations we saw that 
$\intl{\pra}\neq \ilm$.
\end{proof}

\section{Future research}
We have seen that adding a constant for \isig{1} to \pra is sufficient to obtain a non-trivial provability logic. By a Theorem of Leivant it is known that 
$\isig{1} \equiv <2>_{\ea}\top$. An interesting fragment to consider next for \pra would be the closed fragment together with the set of constants 
\[
\{   (<1>_{\ea}<2>_{\ea})^n\top \mid n\in \omega\}
\]
or variants thereof.


\section{Acknowledgements}
We would like to thank Lev Beklemishev, Dick de Jongh and Albert Visser for fruitful comments and discussions.


\end{document}